\documentclass{svjour3}

\smartqed

\usepackage{amsmath}
\usepackage{amssymb}
\usepackage{cancel}
\usepackage{xfrac}
\usepackage{multirow}
\usepackage{natbib}
\usepackage[breaklinks,plainpages=false,colorlinks=true,linkcolor=blue,
  anchorcolor=green,citecolor=blue,filecolor=blue,urlcolor=blue]{hyperref}

\usepackage{fullpage}
\newcommand{\mref}[1]{(\ref{#1})}
\newcommand{\BIC}{\operatorname{BIC}}
\newcommand{\dd}{\,d}
\newcommand{\G}{\mathcal{G}}
\newcommand{\Gp}{\mathcal{G}^{+}}
\newcommand{\Gm}{\mathcal{G}^{-}}
\newcommand{\E}{\operatorname{E}}
\renewcommand{\H}{\operatorname{H}}
\newcommand{\HG}{\operatorname{H}^{\G}}

\newcommand{\D}{\mathcal{D}}
\newcommand{\B}{\mathcal{B}}
\newcommand{\X}{\mathbf{X}}
\newcommand{\PX}{\Pi_{X}}
\newcommand{\PXGm}{\Pi_{X}^{\Gm}}
\newcommand{\PXGp}{\Pi_{X}^{\Gp}}
\newcommand{\PXi}{\Pi_{X_i}}
\newcommand{\PXiG}{\Pi_{X_i}^{\G}}
\newcommand{\PXiGm}{\Pi_{X_i}^{\Gm}}
\newcommand{\PXiGp}{\Pi_{X_i}^{\Gp}}
\newcommand{\XP}{X \given \PX}
\newcommand{\XPGm}{X \given \PXGm}
\newcommand{\XPGp}{X \given \PXGp}
\newcommand{\XPi}{X_i \given \PXi}
\newcommand{\XPiG}{X_i \given \PXiG}
\newcommand{\XPiGm}{X_i \given \PXiGm}
\newcommand{\XPiGp}{X_i \given \PXiGp}
\newcommand{\T}{\Theta_{X_i}}
\newcommand{\TT}{\T \given \PXi}
\newcommand{\BD}{\operatorname{BD}}
\newcommand{\BDeu}{\operatorname{BDeu}}
\newcommand{\BDs}{\operatorname{BDs}}
\newcommand{\indep}{\perp\hspace{-0.21cm}\perp}
\newcommand{\given}{\operatorname{|}}
\newcommand{\Prob}{\operatorname{P}}
\newcommand{\tq}{\tilde{q}_i}
\newcommand{\aij}{\alpha_{ij}}
\newcommand{\aijk}{\alpha_{ijk}}
\newcommand{\piijk}{\pi_{ik \given j}}
\newcommand{\pijk}{p_{ik \given j}}
\newcommand{\jpo}{j: n_{ij} > 0}
\newcommand{\fh}{\sfrac{1}{2}}
\newcommand{\Dir}{\mathit{Dirichlet}}
\newcommand{\Mul}{\mathit{Multinomial}}
\newcommand{\bpi}{\boldsymbol\pi}
\newcommand{\api}{\boldsymbol\alpha}
\newcommand{\e}[1]{\times 10^{#1}}
\newcommand{\dEP}[1]{d_\mathrm{EP}^{(#1)}}

\journalname{Behaviormetrika}

\begin{document}

\title{Dirichlet Bayesian Network Scores and the Maximum Relative Entropy
  Principle}

\author{Marco Scutari}

\institute{M. Scutari \at
           Department of Statistics, University of Oxford,
           24--29 St. Giles', Oxford OX1 3LB, United Kingdom.
           \email{scutari@stats.ox.ac.uk}
}

\date{Received: \ldots / Accepted: \ldots}

\maketitle

\begin{abstract}%
  A classic approach for learning Bayesian networks from data is to identify a
  \emph{maximum a posteriori} (MAP) network structure. In the case of discrete
  Bayesian networks, MAP networks are selected by maximising one of several
  possible Bayesian Dirichlet (BD) scores; the most famous is the \emph{Bayesian
  Dirichlet equivalent uniform} (BDeu) score from \citet{heckerman}. The key
  properties of BDeu arise from its uniform prior over the parameters of each
  local distribution in the network, which makes structure learning
  computationally efficient; it does not require the elicitation of prior
  knowledge from experts; and it satisfies score equivalence.

  In this paper we will review the derivation and the properties of BD scores,
  and of BDeu in particular, and we will link them to the corresponding entropy
  estimates to study them from an information theoretic perspective. To this
  end, we will work in the context of the foundational work of \citet{caticha},
  who showed that Bayesian inference can be framed as a particular case of the
  maximum relative entropy principle. We will use this connection to show that
  BDeu should not be used for structure learning from sparse data, since it
  violates the maximum relative entropy principle; and that it is also
  problematic from a more classic Bayesian model selection perspective, because
  it produces Bayes factors that are sensitive to the value of its only
  hyperparameter. Using a large simulation study, we found in our previous work
  \citep{jmlr16} that the Bayesian Dirichlet sparse (BDs) score seems to provide
  better accuracy in structure learning; in this paper we further show that
  BDs does not suffer from the issues above, and we recommend to use it for
  sparse data instead of BDeu. Finally, will show that these issues are in fact
  different aspects of the same problem and a consequence of the distributional
  assumptions of the prior.


\end{abstract}

\section{Introduction and Background}

Bayesian networks \citep[BNs;][]{pearl,koller} are probabilistic graphical
models based on a directed acyclic graph (DAG) $\G$ whose nodes are associated
with a set of random variables $\X = \{X_1, \ldots, X_N\}$ following some
distribution $\Prob(\X)$. (The two are referred to interchangeably.) Formally,
$\G$ is defined as an \emph{independency map} of $\Prob(\X)$ such that
\begin{equation*}
  \X_A \indep_G \X_B \given \X_C \Longrightarrow
    \X_A \indep_P \X_B \given \X_C,
\end{equation*}
where $\X_A$, $\X_B$ and $\X_C$ are disjoint subsets of $\X$. In other words,
\textit{graphical separation} (denoted $\indep_G$, and called
\emph{d-separation} in this context) between two nodes in $\G$ implies the
\textit{conditional independence} (denoted $\indep_P$) of the corresponding
variables in $\X$. Two nodes linked by an arc cannot be graphically separated;
hence the arcs of $\G$ represent \emph{direct dependencies} between the
variables they are incident on, as opposed to \emph{indirect dependencies} that
are mediated by one or more nodes in $\G$.

A consequence of this definition is the Markov property \citep{pearl}: in the
absence of missing data the \emph{global distribution} of $\X$ decomposes into
\begin{equation}
\label{eq:parents}
  \Prob(\X \given \G) = \prod_{i=1}^N \Prob(\XPiG)
\end{equation}
where the \emph{local distribution} of each node $X_i$ depends only on the
values of its parents $\PXi$ in $\G$ (denoted $\PXiG$). In this paper we will
focus on discrete BNs \citep{heckerman}, in which both $\X$ and the $X_i$ are
multinomial random variables; in particular
\begin{equation*}
  \XPiG \sim \Mul(\TT^\G)
\end{equation*}
where each parameter set $\T$ comprises the conditional probabilities
\begin{equation*}
  \piijk = \Prob(X_i = k \given \PXiG = j)
\end{equation*}
of each value $k$ of $X_i$ given each possible configuration of the values of
$\PXiG$. Other possibilities include Gaussian BNs \citep{heckerman3} and
conditional linear Gaussian BNs \citep{lauritzen}. In Gaussian BNs, $\X$ is
multivariate normal and the (conditional) dependencies between the $X_i$ are
assumed to be linear, leading to
\begin{equation*}
  \XPiG \sim N(\TT^\G)
\end{equation*}
which can be written as a linear regression model of the form
\begin{align*}
  &X_i = \mu_{X_i} + \PXiG \beta_{X_i} + \varepsilon,&
  &\varepsilon \sim N(0, \sigma^2_{X_i})
\end{align*}
or using the partial correlations between $X_i$ and each parent given the rest;
the two parameterisations are equivalent \citep{weatherburn}. Conditional linear
Gaussian BNs combine multinomial and normal variables using mixture of normals,
with discrete variables identifying the components of the mixture.

It is important to note that the decomposition in \mref{eq:parents} does not
uniquely identify a BN; different DAGs can encode the same global distribution,
thus grouping BNs into equivalence classes \citep{chickering} characterised by
the skeleton of $\G$ (its underlying undirected graph) and its v-structures
(patterns of arcs of the type $X_j \rightarrow X_i \leftarrow X_k$, with no arc
between $X_j$ and $X_k$). Intuitively, the direction of arcs that are not part
of v-structures can be reversed without changing the global distribution, just
factorising it in different ways, as long as the new arc directions do not
introduce additional v-structures or cycles in the DAG. As a simple example,
consider
\begin{multline*}
  \underbrace{\Prob(X_i)\Prob(X_j \given X_i)\Prob(X_k \given X_j)}_{X_i \rightarrow X_j \rightarrow X_k} =
  \Prob(X_j, X_i)\Prob(X_k \given X_j) = \\ =
  \underbrace{\Prob(X_i \given X_j)\Prob(X_j)\Prob(X_k \given X_j)}_{X_i \leftarrow X_j \rightarrow X_k}.
\end{multline*}

The task of specifying BNs is called \emph{learning} and can be performed using
either data, prior expert knowledge on the phenomenon being modelled or both. The
latter has been shown to provide very good results in a variety of applications,
and should be preferred if it is feasible to elicit prior information from
experts \citep{csprior,mukherjee}. Learning BNs from data is usually performed
in an inherently Bayesian fashion by maximising
\begin{align}
\label{eq:lproc}
  &\underbrace{\Prob(\B \given \D) =
  \Prob(\G, \Theta \given \D)}_{\text{learning}}& &=&
    &\underbrace{\Prob(\G \given \D)}_{\text{structure learning}}& &\cdot&
    &\underbrace{\Prob(\Theta \given \G, \D)}_{\text{parameter learning}},
\end{align}
where $\D$ is a sample from $\X$ and $\B = (\G, \Theta)$ is a BN with DAG $\G$
and parameter set $\Theta = \bigcup_{i=1}^N \T$. Structure learning consists in
finding the DAG $\G$ that encodes the dependence structure of the data;
parameter learning involves the estimation of the parameters $\Theta$ given
$\G$. Expert knowledge can be incorporated in either or both these steps through
the use of informative priors for $\G$ or $\Theta$.

Structure learning can be implemented in several ways, based on many results
from probability, information and optimisation theory; algorithms for this task
can be broadly grouped into constraint-based, score-based and hybrid.

\textit{Constraint-based algorithms} \citep{hiton1,hiton2} use statistical tests
to learn conditional independence relationships from the data and to determine
if the corresponding arcs should be included in $\G$. In order to do that
they assume that $\G$ is \emph{faithful} to $\Prob(\X)$, meaning
\begin{equation*}
  \X_A \indep_G \X_B \given \X_C \Longleftrightarrow
    \X_A \indep_P \X_B \given \X_C;
\end{equation*}
this is a strong assumption that does not hold in a number of real-world
scenarios, which are reviewed in \citet{koski}. Depending on the nature of the
data, conditional independence tests in common use are the mutual information
($G^2$) and Pearson's $\chi^2$ tests for contingency tables (for discrete BNs);
and Fisher's $Z$ and the exact $t$ tests for partial correlations (for Gaussian
BNs); an overview is provided in \citet{crc13}.

\textit{Score-based algorithms} are closer to model selection techniques
developed in classic statistics and information theory. Each candidate network
is assigned a score reflecting its goodness of fit, which is then taken as an
objective function to maximise. This is often done using heuristic optimisation
algorithms, from local search to genetic algorithms \citep{norvig}; but the
availability of computational resources and advances in learning algorithms have
recently made exact learning possible \citep{cutting}. Common choices for the
network score include the Bayesian Information Criterion (BIC) and the marginal
likelihood $\Prob(\G \given \D)$ itself; for an overview see again \citet{crc13}.
We will cover both in more detail for discrete BNs in Section \ref{sec:bd}.

\textit{Hybrid algorithms} use both statistical tests and score functions,
combining the previous two families of algorithms. Their general formulation is
described for BNs in \citet{sparse}; they have proved to be some of the top
performers up to date \citep[see for instance MMHC in][]{mmhc}.

As for parameter learning, the parameters $\T$ can be estimated independently
for each node following \mref{eq:parents} since its parents are assumed to be
known from structure learning. Both maximum likelihood and Bayesian posterior
estimators are in common use, with the latter being preferred due to their
smoothness and superior predictive power \citep{koller}.

In this paper we focus on score-based structure learning in a Bayesian
framework, in which we aim to identify a \emph{maximum a posteriori} (MAP) DAG
$\G$ that directly maximises $\Prob(\G \given \D)$. For discrete BNs, this means
maximising a Bayesian Dirichlet (BD) marginal likelihood: the most common choice
is the \emph{Bayesian Dirichlet equivalent uniform} (BDeu) score from
\citet{heckerman}. We will show that the uniform prior distribution over each
$\T$ that underlies BDeu can be problematic from a Bayesian perspective,
resulting in wildly different Bayes factors (and thus structure learning
outcomes) depending on the value of its only hyperparameter, the imaginary
sample size. We will further investigate this problem from an information
theoretic perspective, on the grounds that Bayesian posterior inference can be
framed as a particular case of the maximum relative entropy principle
\citep[ME;][]{shore,skilling,caticha2}. We find that BDeu is not a reliable
network score when applied to sparse data because it can select overly complex
networks over simpler ones given the same information in the prior and in the
data; and that in the process it violates the maximum relative entropy
principle. That does not appear to be the case for other BD scores, which arise
from different priors.

The paper is organised as follows. In Section \ref{sec:bd} we will review
Bayesian score-based structure learning and BD scores. In Section
\ref{sec:problems} we will focus on BDeu, covering its underlying assumptions
and issues reported in the literature. In particular, we will show with simple
examples that BDeu can produce Bayes factors that are sensitive to the choice of
its only hyperparameter, the imaginary sample size. In Section \ref{sec:maxent}
we will derive the posterior expected entropy associated with a DAG $\G$, which
we will further explore in Section \ref{sec:expent}. Finally, in Section
\ref{sec:bd-maxent} we will analyse BDeu using ME, and we will compare its
behaviour with that of other BD scores.

\section{Bayesian Dirichlet Marginal Likelihoods}
\label{sec:bd}

Score-based structure learning, when performed in a Bayesian framework, aims at
finding a DAG $\G$ that has the highest posterior probability
$\Prob(\G \given \D)$. Starting from \mref{eq:lproc}, using Bayes' theorem we
can write
\begin{equation*}
  \Prob(\G \given \D) \propto \Prob(\G)\Prob(\D \given \G) =
    \Prob(\G)\int \Prob(\D \given \G, \Theta) \Prob(\Theta \given \G) \dd\Theta
\end{equation*}
where $\Prob(\G)$ is the prior distribution over the space of the DAGs spanning
the variables in $\X$ and $\Prob(\D \given \G)$ is the marginal likelihood of
the data given $\G$ averaged over all possible parameter sets $\Theta$.
$\Prob(\G)$ is often taken to be uniform so that it simplifies when comparing
DAGs; we will do the same in this paper for simplicity while noting that other
default priors may lead to more accurate structure learning of sparse DAGs
\citep[\emph{e.g.}][]{jmlr16}. Using \mref{eq:parents} we can then decompose
$\Prob(\D \given \G)$ into one component for each node as follows:
\begin{equation}
  \Prob(\D \given \G) = \prod_{i=1}^N \Prob(\XPiG) =
    \prod_{i=1}^N \left[ \int \Prob(\XPiG, \T)
    \Prob(\TT^\G) \dd\T \right].
\label{eq:structlearn}
\end{equation}
In the case of discrete BNs, we assume  $\XPiG \sim \Mul(\TT^\G)$ where the
parameters $\TT^\G$ are the conditional probabilities $\piijk = \Prob(X_i = k
\given \PXiG = j)$. We then assume a conjugate prior $\TT^\G \sim \Dir(\aijk)$,
$\sum_{jk} \aijk = \alpha_i > 0$ to obtain the closed-form posterior
$\Dir(\aijk + n_{ijk})$ which we use to estimate the $\piijk$ from the counts
$n_{ijk}, \sum_{ijk} n_{ijk} = n$ observed in $\D$. $\alpha_i$ is known as the
\emph{imaginary} or \emph{equivalent sample size} and determines how much weight
is assigned to the prior in terms of the size of an imaginary sample supporting
it.

Further assuming \emph{positivity} ($\piijk > 0$), \emph{parameter independence}
($\piijk$ for different parent configurations are independent),
\emph{parameter modularity} ($\piijk$ associated with different nodes are
independent) and \emph{complete data}, \citet{heckerman} derived
a closed form expression for \mref{eq:structlearn}, known as the \emph{Bayesian
Dirichlet} (BD) family of scores:
\begin{multline}
  \BD(\G, \D; \api) =
  \prod_{i=1}^N \BD\left(\XPiG; \alpha_i\right) \\ =
  \prod_{i=1}^N \prod_{j = 1}^{q_i}
    \left[
      \frac{\Gamma(\alpha_{ij})}{\Gamma(\alpha_{ij} + n_{ij})}
      \prod_{k=1}^{r_i} \frac{\Gamma(\aijk + n_{ijk})}{\Gamma(\aijk)}
    \right]
\label{eq:bd}
\end{multline}
where:
\begin{itemize}
  \item $r_i$ is the number of states of $X_i$
  \item $q_i$ is the number of possible configurations of values of $\PXiG$,
    taken to be equal to $1$ if $X_i$ has no parents;
  \item $n_{ij} = \sum_{k = 1}^{r_i} n_{ijk}$;
  \item $\aij = \sum_{k = 1}^{r_i} \aijk$;
  \item and $\api = \{\alpha_1, \ldots, \alpha_N\}$,
    $\alpha_i = \sum_{j = 1}^{q_i} \aij$ are the imaginary sample
    sizes associated with each $X_i$.
\end{itemize}
Various choices for $\aijk$ produce different priors and the corresponding
scores in the BD family:
\begin{itemize}
  \item for $\aijk = 1$ we obtain the K2 score from \citet{k2};
  \item for $\aijk = \fh$ we obtain the BD score with Jeffrey's prior
    \citep[BDJ;][]{suzuki16};
  \item for $\aijk = \alpha / (r_i q_i)$ we obtain the BDeu score from
    \citet{heckerman}, which is the most common choice in the BD family and
    has $\alpha_i = \alpha$ for all $X_i$;
  \item for $\aijk = \alpha / (r_i \tq)$, where $\tq$ is the number of
    $\PXiG$ such that $n_{ij} > 0$, we obtain the BD sparse (BDs) score recently
    proposed in \citet{jmlr16};
  \item for the set $\aijk^s = s / (r_i q_i)$, $s \in S_L = \{2^{-L}, 2^{-L+1},
    \ldots, 2^{L-1}, 2^{L}\}$, $L \in \mathbb{N}$ we obtain the locally averaged
    BD score (BDla) from \citet{bdla}.
\end{itemize}
BDeu is the only score-equivalent BD score \citep{chickering}, that is, it is
the only BD score that takes the same value for DAGs in the same equivalence
class. This property makes BDeu the preferred score when arcs are given causal
interpretation \citep{causality}, and their directions have a meaningful
interpretation beyond allowing to decompose the $\Prob(\X)$ into the
$\Prob(\XPiG)$. BDs is only asymptotically score-equivalent because it converges
to BDeu when $n \to \infty$ and the positivity assumption holds. The BIC score,
defined as
\begin{equation}
  \BIC(\G, \D) = \sum_{i = 1}^N \BIC\left(\XPiG\right)
    = \sum_{i = 1}^N \left[ \log\Prob(\XPiG) - \frac{\log n}{2}q_i (r_i - 1)\right]
\end{equation}
is also score-equivalent and it converges to $\log\BDeu$ as $n \to \infty$. In
the case of discrete BNs, maximising BIC corresponds to selecting the BN with
the \emph{minimum description length} \citep[MDL;][]{mdl}.

\section{BDeu and Bayesian Model Selection}
\label{sec:problems}

The uniform prior associated with BDeu has been justified by the lack of prior
knowledge on the $\T$, as well as its computational simplicity and
score equivalence; and it was widely assumed to be non-informative
\citep[\emph{e.g.}][]{silander,heckerman}.

However, there is increasing evidence that this prior is far from
non-informative and that it has a strong impact on the accuracy of the learned
DAGs, making its use on real-world data problematic. \citet{silander} showed via
simulation that the MAP DAGs selected using BDeu are highly sensitive to the
choice of $\alpha$. Even for ``reasonable'' values such as $\alpha \in [1, 20]$,
they obtained DAGs with markedly different number of arcs, and they showed that
large values of $\alpha$ tend to produce DAGs with more arcs. This is
counter-intuitive because a larger $\alpha$ would normally imply stronger
regularisation and would be expected to produce sparser DAGs. \citet{jaakkola}
similarly showed that the number of arcs in the MAP DAG is determined by a
complex interaction between $\alpha$ and $\D$; in the limits $\alpha \to 0$ and
$\alpha \to \infty$ it is possible to obtain both very sparse and very dense
DAGs. (We informally define $\G$ to be \emph{sparse} if $|A| = O(N)$, typically
with $|A| < 5N$; a \emph{dense} $\G$, on the other hand, has a relatively large
$|A|$ compared to $N$.) In particular, for small values of $\alpha$ and/or
sparse data (that is, discrete data for which we observe a small subset of the
possible combinations of the values of the $X_i$), $\aijk \to 0$ and
\begin{equation}
  \lim_{\aijk \to 0} \BDeu(\G, \D; \alpha) - \alpha^{\dEP{\G}} = 0
\label{eq:convergence}
\end{equation}
where $\dEP{\G}$ is the effective number of parameters of the model, defined as
\begin{equation*}
  \dEP{\G} =  \sum_{i = i}^N \dEP{X_i, \G} = \sum_{i = i}^N
    \left[ \sum_{j = 1}^{q_i} \tilde{r}_{ij} - \tq \right];
\end{equation*}
$\tilde{r}_{ij}$ is the number of positive counts $n_{ijk}$ in the $j$th
configuration of $\PXiG$ and $\tq$ is the number of configurations in which at
least one $n_{ijk}$ is positive.

This was then used to prove that the Bayes factor
\begin{equation}
  \frac{\Prob(\D \given \Gp)}{\Prob(\D \given \Gm)} =
  \frac{\BDeu(\XPiGp; \alpha)}{\BDeu(\XPiGm; \alpha)}
  \to
  \left\{
  \begin{aligned}
  0 & & \text{if $d_\mathrm{EDF} > 0$}\\
  +\infty & & \text{if $d_\mathrm{EDF} < 0$}
  \end{aligned}
  \right.
\label{eq:bdeu-ratio}
\end{equation}
for two DAGs $\Gm$ and $\Gp$ that differ only by the inclusion of a single
parent for $X_i$. The effective degrees of freedom
$d_\mathrm{EDF}$ are defined as $\dEP{\Gp} - \dEP{\Gm}$. The practical
implication of this result is that, when we compare two DAGs using their BDeu
 scores, a large number of zero counts will force $d_\mathrm{EDF}$ to be negative
and favour the inclusion of additional arcs (since $\BDeu(\XPiGp; \alpha) \gg
\BDeu(\XPiGm; \alpha)$).
But that in turn makes $d_\mathrm{EDF}$ even more negative, quickly leading to
overfitting. Furthermore, \citet{jaakkola} argued that BDeu can be rather
unstable for ``medium-sized'' data and small $\alpha$, which is a very common
scenario.

\citet{alphastar} approached the problem from a different perspective and
derived an analytic approximation for the ``optimal'' value of $\alpha$ that
maximises predictive accuracy, further suggesting that the interplay between
$\alpha$ and $\D$ is controlled by the skewness of the $\T$ and by the strength
of the dependence relationships between the nodes. Skewed $\T$ result in some
$\piijk$ being smaller than others, which in turn makes sparse data sets more
likely; hence the problematic behaviour described in \citet{jaakkola} and
reported above. Most of these results have been analytically confirmed more
recently by \citet{ueno,ueno2}. The key insight provided by the latter paper is
that we can decompose the BDeu into a \emph{likelihood term} that depends on the
data and a \emph{prior term} that does not:
\begin{multline*}
  \log\BDeu(\XPiG; \alpha) =
    \underbrace{\sum_{j=1}^{q_i}
      \left( \log\Gamma(\aij) - \sum_{k=1}^{r_i} \log\Gamma(\aijk) \right)
    }_{\text{prior term}} + \\
    \underbrace{\sum_{j=1}^{q_i}
      \left(\sum_{k=1}^{r_i} \log\Gamma(\aijk + n_{ijk}) - \log\Gamma(\aij + n_{ij}) \right)
    }_{\text{likelihood term}}.
\end{multline*}
Then if $\aijk < 1$ (that is, $\alpha < r_i q_i$) the prior term can be
approximated by
\begin{equation*}
  \sum_{j=1}^{q_i}
    \left( \log\Gamma(\aij) - \sum_{k=1}^{r_i} \log\Gamma(\aijk) \right) \approx
  q_i (r_i - 1) \log \aijk
\end{equation*}
and quickly dominates the likelihood term, penalising complex BNs as the number
of parameters increases, which explains why BDeu selects empty DAGs in the limit
of $\aijk \to 0$. On the other hand, if all $\aijk > 1$ then the prior term can
be approximated by
\begin{equation*}
  \sum_{j=1}^{q_i}
    \left( \log\Gamma(\aij) - \sum_{k=1}^{r_i} \log\Gamma(\aijk) \right) \approx
  \alpha\log r_i + \frac{1}{2} q_i(r_i - 1)\log\frac{\aijk}{2\pi},
\end{equation*}
leading BDeu to select a complete DAG when $\aijk \to \infty$ (and therefore
$\alpha \to \infty$) as previously reported in \citet{silander}.

As for the likelihood term, \citet{ueno2} notes that if $\alpha + n$ is
sufficiently small, that is, for sparse samples and small imaginary sample sizes,
\begin{equation*}
  \sum_{j=1}^{q_i}
      \left(\sum_{k=1}^{r_i} \log\Gamma(\aijk + n_{ijk}) - \log\Gamma(\aij + n_{ij}) \right)
  \approx - q_i(r_i - 1)\log\aijk.
\end{equation*}
Hence if some $n_{ijk} = 0$, the change of the likelihood term dominates the
prior term and BDeu adds extra arcs, leading to dense DAGs. On the other hand,
if $\alpha + n$ is sufficiently large, $\alpha$ actually acts as an imaginary
sample supporting the uniform distribution of the parameters assumed in the
prior. This explains the observations in \citet{alphastar}: the optimal $\alpha$
should be large when the empirical distribution of the $\T$ is uniform because
the prior is correct; and it should be small when the empirical distribution of
$\T$ is skewed so that the prior can be quickly dominated. This is also the
source of the sensitivity of BDeu to the choice of $\alpha$ reported in
\citet{jaakkola}.

Finally, \citet{suzuki16} studied the asymptotic properties of BDeu by
contrasting it with BDJ. He found that BDeu is not \emph{regular} in the sense
that it may learn DAGs in a way that is not consistent with either the MDL
principle (through BIC) or the ranking of those DAGs given by their entropy.
Whether this happens depends on the values of the underlying $\piijk$, even if
the positivity assumption holds and if $n$ is large. This agrees with the
observations in \citet{ueno}, who also observed that BDeu is not necessarily
consistent for any finite $n$, but only asymptotically for $n \to \infty$.

Around the same time, a possible solution to these problems was proposed by
\citet{jmlr16} in the form of BDs. \citet{jmlr16} starts from the consideration
that if the sample $\D$ is sparse, some configurations of the variables will not
be observed; it may be that the sample size is small and those configurations
have low probability, or it may be that $\X$ violates the positivity assumption
($\piijk = 0$ for some $i,j,k$). As a result, we may be unable to observe all
the configurations of (say) $\PXiGp$ in the data. Then the corresponding
$n_{ij} = 0$ and
\begin{multline*}
  \BDeu(\XPiGp; \alpha) = \\
    \prod_{j : n_{ij} = 0} \left[
    \cancel{
      \frac{\Gamma(r_i \aijk)}{\Gamma(r_i \aijk)}
      \prod_{k=1}^{r_i} \frac{\Gamma(\aijk)}{\Gamma(\aijk)}
    }
    \right]
    \prod_{\jpo} \left[
      \frac{\Gamma(r_ i \aijk)}{\Gamma(r_i \aijk + n_{ij})}
      \prod_{k=1}^{r_i} \frac{\Gamma(\aijk + n_{ijk})}{\Gamma(\aijk)}
    \right].
\end{multline*}
The effective imaginary sample size, defined as the sum of the $\aijk$ appearing
in terms that do not simplify (and thus contribute to the value of BDeu),
decreases to $\sum_{\jpo} \aijk = \alpha (\tq / q_i) < \alpha$, where $\tq < q_i$
is the number of parent configurations that are actually observed in $\D$. In
other words, BDeu is computed with an imaginary sample size of
$\alpha (\tq / q_i)$ instead of $\alpha$, and the remaining
$\alpha (q_i - \tq) / q_i$ is effectively lost. This may lead to comparing DAGs
with marginal likelihoods computed from different priors, which is incorrect
from a Bayesian perspective. In order to prevent this from happening,
\citet{jmlr16} replaced the prior of BDeu with
\begin{align*}
  &\alpha_{ijk} = \left\{
    \begin{aligned}
      &\alpha / (r_i \tq)& & \text{if $n_{ij} > 0$} \\
      &0                         & & \text{otherwise.}
    \end{aligned}
  \right.& &\text{where}&
  &\tq = \{ \text{number of $\PXiGp$ such that $n_{ij} > 0$} \}.
\end{align*}
thus obtaining
\begin{equation}
  \BDs(\XPiGp; \alpha) =
    \prod_{j : n_{ij} > 0}
    \left[
      \frac{\Gamma(r_i \aijk)}{\Gamma(r_i \aijk + n_{ij})}
      \prod_{k=1}^{r_i} \frac{\Gamma(\aijk + n_{ijk})}{\Gamma(\aijk)}
    \right].
\label{eq:bds}
\end{equation}
A large simulation study showed BDs to be more accurate than BDeu in learning BN
structures without any loss in predictive power.

In addition to all these issues, we also find that BDeu produces Bayes factors
that are sensitive to the choice of $\alpha$. (The fact that BDeu is sensitive
to the value of $\alpha$ does not necessarily imply that the Bayes factor is
sensitive itself.) In order to illustrate this instability and the other results
presented in the section we consider the simple examples below.

\begin{figure}[t]
\begin{center}
  \includegraphics[width=0.9\textwidth]{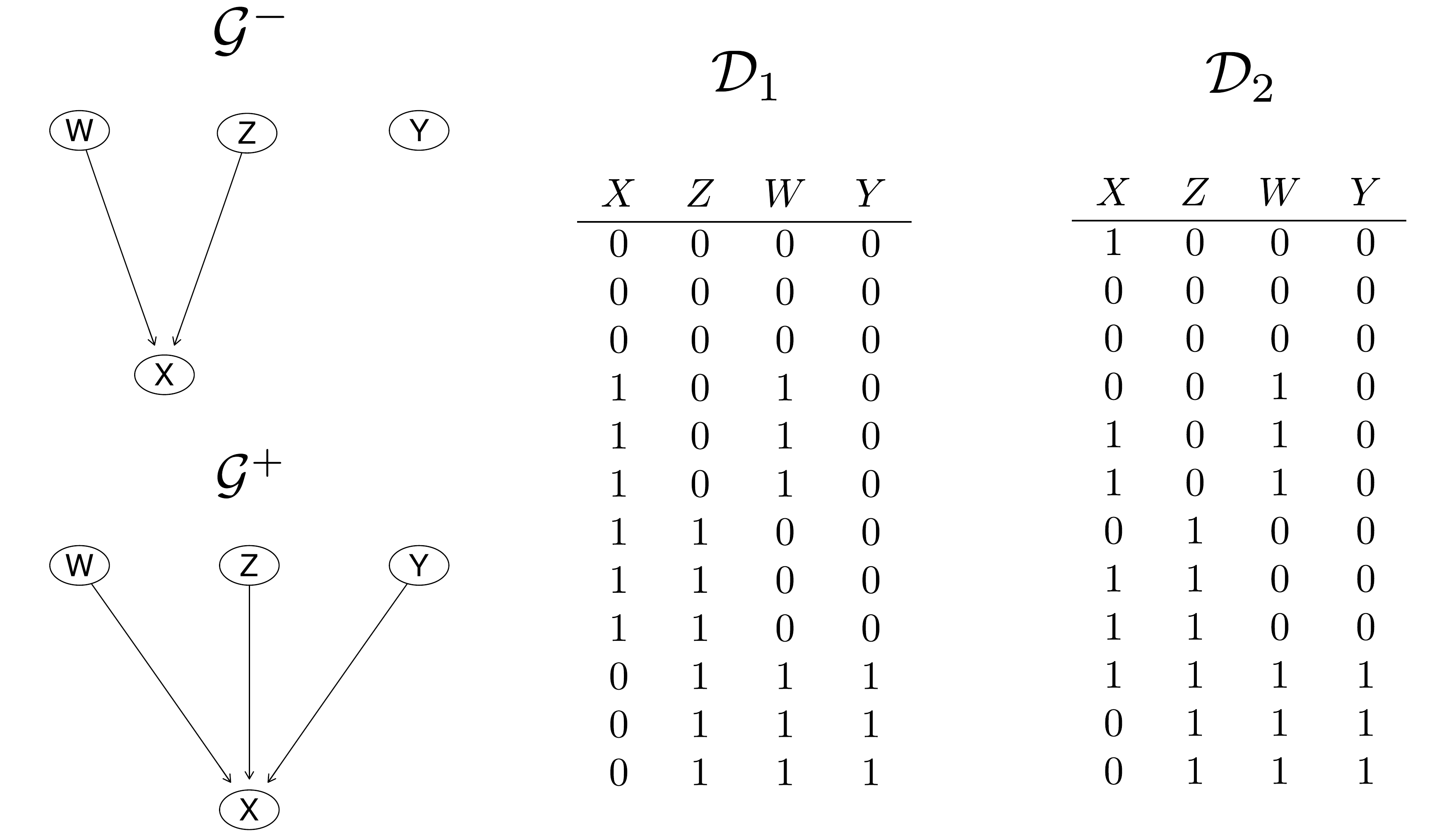}
  \caption{DAGs and data sets used in Examples \ref{ex:singular} and
    \ref{ex:nonsingular}. The DAGs $\Gm$ and $\Gp$ are used in both examples.
    Example \ref{ex:singular} uses data set $\D_1$, while Example
    \ref{ex:nonsingular} uses $\D_2$; the latter is a modified version of the
    former, which is originally from \citet{suzuki16}.}
  \label{fig:examples}
\end{center}
\end{figure}

\begin{example}
\label{ex:singular}
Consider the DAGs $\Gm$ and $\Gp$ and the data set $\D_1$ in Figure
\ref{fig:examples}, originally from \citet{suzuki16}. The sample frequencies
$n_{ijk}$ for $\XP^{\Gm} = \{ Z, W \}$ are:
\begin{center}
\begin{tabular}{rr|cccc}
  \multicolumn{2}{r|}{$Z, W$} & $0, 0$ & $1, 0$ & $0, 1$ & $1, 1$ \\
  \hline
  \multirow{2}{*}{$X$}
    & $0$ & $3$ & $0$ & $0$ & $3$ \\
    & $1$ & $0$ & $3$ & $3$ & $0$
\end{tabular}
\end{center}
and those for $\XP^{\Gp} = \XP^{\Gm} \cup \{ Y \}$ are as follows.
\begin{center}
\begin{tabular}{rr|cccccccc}
    \multicolumn{2}{r|}{$Z, W, Y$} & $0, 0, 0$ & $1, 0, 0$ & $0, 1, 0$ & $1, 1, 0$
      & $0, 0, 1$ & $1, 0, 1$ & $0, 1, 1$ & $1, 1, 1$ \\
  \hline
  \multirow{2}{*}{$X$}
    & $0$ & $3$ & $0$ & $0$ & $0$ & $0$ & $0$ & $0$ & $3$ \\
    & $1$ & $0$ & $3$ & $3$ & $0$ & $0$ & $0$ & $0$ & $0$
\end{tabular}
\end{center}

The conditional distributions of $\XP^{\Gm}$ and $\XP^{\Gp}$ are both singular,
and in $\XP^{\Gp}$ we only observe $4$ parent configurations out of $8$.
Furthermore, the observed conditional distributions for those parent
configurations are identical to the $4$ conditional distributions in $\XP^{\Gm}$,
since the $n_{ijk}$ are the same. We can then argue that $\XP^{\Gp}$ does not
fit $\D_1$ any better than $\XP^{\Gm}$, and it does not capture any additional
information from the data.

However, if we take $\alpha = 1$ in BDeu, then $\aijk = 1/8$ for $\Gm$ and
$\aijk = 1/16$ for $\Gp$, leading to
\begin{align*}
  &\BDeu\left(\XP^{\Gm}; 1\right) =
    \left( \frac{\Gamma(\sfrac{1}{4})}{\Gamma(\sfrac{1}{4}+3)}
    \left[ \frac{\Gamma(\sfrac{1}{8} + 3)}{\Gamma(\sfrac{1}{8})} \cdot
           \cancel{\frac{\Gamma(\sfrac{1}{8})}{\Gamma(\sfrac{1}{8})}} \right] \right)^4
    = 0.0326, \\
  &\BDeu\left(\XP^{\Gp}; 1\right) =
    \left( \frac{\Gamma(\sfrac{1}{8})}{\Gamma(\sfrac{1}{8}+3)}
    \left[ \frac{\Gamma(\sfrac{1}{16} + 3)}{\Gamma(\sfrac{1}{16})} \cdot
           \cancel{\frac{\Gamma(\sfrac{1}{16})}{\Gamma(\sfrac{1}{16})}} \right] \right)^4
    = 0.0441.
\end{align*}
If we choose the DAG with the highest BDeu score, we prefer $\Gp$ to $\Gm$
despite all the considerations we have just made on the data. This is not the
case if we use BDs, which does not show a preference for either $\Gm$ or $\Gp$
because $\aijk = 1/8$ for both $\XP^{\Gm}$ and $\XP^{\Gp}$:
\begin{multline*}
  \BDs\left(\XP^{\Gm}; 1\right) =
  \BDs\left(\XP^{\Gp}; 1\right) = \\
    \left( \frac{\Gamma(\sfrac{1}{4})}{\Gamma(\sfrac{1}{4}+3)}
    \left[ \frac{\Gamma(\sfrac{1}{8} + 3)}{\Gamma(\sfrac{1}{8})} \cdot
           \cancel{\frac{\Gamma(\sfrac{1}{8})}{\Gamma(\sfrac{1}{8})}} \right] \right)^4
    = 0.0326.
\end{multline*}
The same holds for BDJ, and in general for any BD score with a constant $\aijk$.
Comparing the expressions above, it is apparent that the only difference between
them is the value of $\aijk$, which is a consequence of the different number of
configurations of $\PX^{\Gm}$ and $\PX^{\Gp}$.

The Bayes factors for BDeu and BDs are shown for $\alpha \in [10^{-4}, 10^{4}]$
in the left panel of Figure \ref{fig:bf}. The former converges to $1$ for both
$\aijk \to 0$ and $\aijk \to \infty$, but varies between $1$ and $2.5$ for
finite $\alpha$; whereas the latter is equal to $1$ for all values of $\alpha$,
never showing a preference for either $\Gm$ or $\Gp$. The Bayes factor for BDeu
does not diverge nor converge to zero, which is consistent with
\mref{eq:bdeu-ratio} from \citet{jaakkola} as $\dEP{\Gp} - \dEP{\Gm} = 0 - 0 = 0$.
However, it varies most quickly for $\alpha \in [1, 10]$, exactly the range of
the most common values used in practical applications. This provides further
evidence supporting the conclusions of \citet{jaakkola}, \citet{alphastar} and
\citet{silander}.

Finally, if we consider which DAG would be preferred according to the MDL
principle, we can see that BIC (unlike BDeu, like BDs) does not express a
preference for either DAG:
\begin{equation*}
  \BIC\left(\XP^{\Gm}\right) = \log\Prob\left(\XP^{\Gm}\right) - 0 =
                    \log\Prob\left(\XP^{\Gp}\right) - 0 = \BIC\left(\XP^{\Gp}\right)
\end{equation*}
which agrees with \citet{suzuki16}'s observation that BDeu violates the MDL
principle. \qed
\end{example}

\begin{example}
\label{ex:nonsingular}
Consider another simple example, inspired by Example \ref{ex:singular}, based on
the data set $\D_2$ and the DAGs $\Gm$, $\Gp$ shown in Figure \ref{fig:examples}.
The sample frequencies ($n_{ijk}$) for $\XP^{\Gm}$ are:
\begin{center}
\begin{tabular}{rr|cccc}
    & & \multicolumn{4}{|c}{$Z, W$} \\
    & & $0, 0$ & $1, 0$ & $0, 1$ & $1, 1$ \\
  \hline
  \multirow{2}{*}{$X$}
    & $0$ & $2$ & $1$ & $1$ & $2$ \\
    & $1$ & $1$ & $2$ & $2$ & $1$
\end{tabular}
\end{center}
and those for $\XP^{\Gp}$ are as follows.
\begin{center}
\begin{tabular}{rr|cccccccc}
    & & \multicolumn{8}{c}{$Z, W, Y$} \\
    & & $0, 0, 0$ & $1, 0, 0$ & $0, 1, 0$ & $1, 1, 0$
      & $0, 0, 1$ & $1, 0, 1$ & $0, 1, 1$ & $1, 1, 1$ \\
  \hline
  \multirow{2}{*}{$X$}
    & $0$ & $2$ & $1$ & $1$ & $0$ & $0$ & $0$ & $0$ & $2$ \\
    & $1$ & $1$ & $2$ & $2$ & $0$ & $0$ & $0$ & $0$ & $1$
\end{tabular}
\end{center}

As in Example \ref{ex:singular}, $4$ parent configurations out of $8$ are not
observed in $\Gp$ and the other $4$ have $n_{ijk}$ that are the same as those
arising from $\Gm$. The resulting conditional distributions, however, are not
singular. If we take again $\alpha = 1$, the BDeu scores for $\Gm$ and $\Gp$ are
different but this time $\Gm$ has the highest score:
\begin{align*}
  \BDeu\left(\XP^{\Gm}; 1\right) &=
    \left( \frac{\Gamma(\sfrac{1}{4})}{\Gamma(\sfrac{1}{4}+3)}
    \left[ \frac{\Gamma(\sfrac{1}{8} + 2)}{\Gamma(\sfrac{1}{8})} \cdot
           \frac{\Gamma(\sfrac{1}{8} + 1)}{\Gamma(\sfrac{1}{8})} \right] \right)^4
              = 3.906\e{-7}, \\
  \BDeu\left(\XP^{\Gp}; 1\right) &=
    \left( \frac{\Gamma(\sfrac{1}{8})}{\Gamma(\sfrac{1}{8}+3)}
    \left[ \frac{\Gamma(\sfrac{1}{16} + 2)}{\Gamma(\sfrac{1}{16})} \cdot
           \frac{\Gamma(\sfrac{1}{16} + 1)}{\Gamma(\sfrac{1}{16})} \right] \right)^4
                     = 3.721\e{-8}.
\end{align*}
On the other hand, in BDs $\aijk = 1/8$ for both DAGs, so they have the same
score:
\begin{multline*}
  \BDs\left(\XP^{\Gm}; 1\right) =
  \BDs\left(\XP^{\Gp}; 1\right) = \\
    \left( \frac{\Gamma(\sfrac{1}{4})}{\Gamma(\sfrac{1}{4}+3)}
    \left[ \frac{\Gamma(\sfrac{1}{8} + 3)}{\Gamma(\sfrac{1}{8})} \cdot
           \cancel{\frac{\Gamma(\sfrac{1}{8})}{\Gamma(\sfrac{1}{8})}} \right] \right)^4
    = 3.906\e{-7}.
\end{multline*}
BDeu once more assigns different scores to $\Gm$ and $\Gp$ despite the fact
that the observed conditional distributions in $\XP^{\Gm}$ and $\XP^{\Gp}$ are
the same, while BDs does not.

The Bayes factors for BDeu and BDs are shown in the right panel of Figure
\ref{fig:bf}. BDeu results in wildly different values depending on the choice of
$\alpha$, with Bayes factors that vary between $0.05$ and $1$ for small changes
of $\alpha \in [1, 10]$; BDs always gives a Bayes factor of $1$. Again
$\dEP{\Gp} - \dEP{\Gm} = 4 - 4 = 0$, which agrees with the fact that the Bayes
factor for BDeu does not diverge or converge to zero; and $\Gm$ and $\Gp$ have
the same BIC score, so BDeu (but not BDs) violates the MDL principle in this
example as well. \qed
\end{example}

\begin{figure}[t]
\begin{center}
\includegraphics[width=0.9\textwidth]{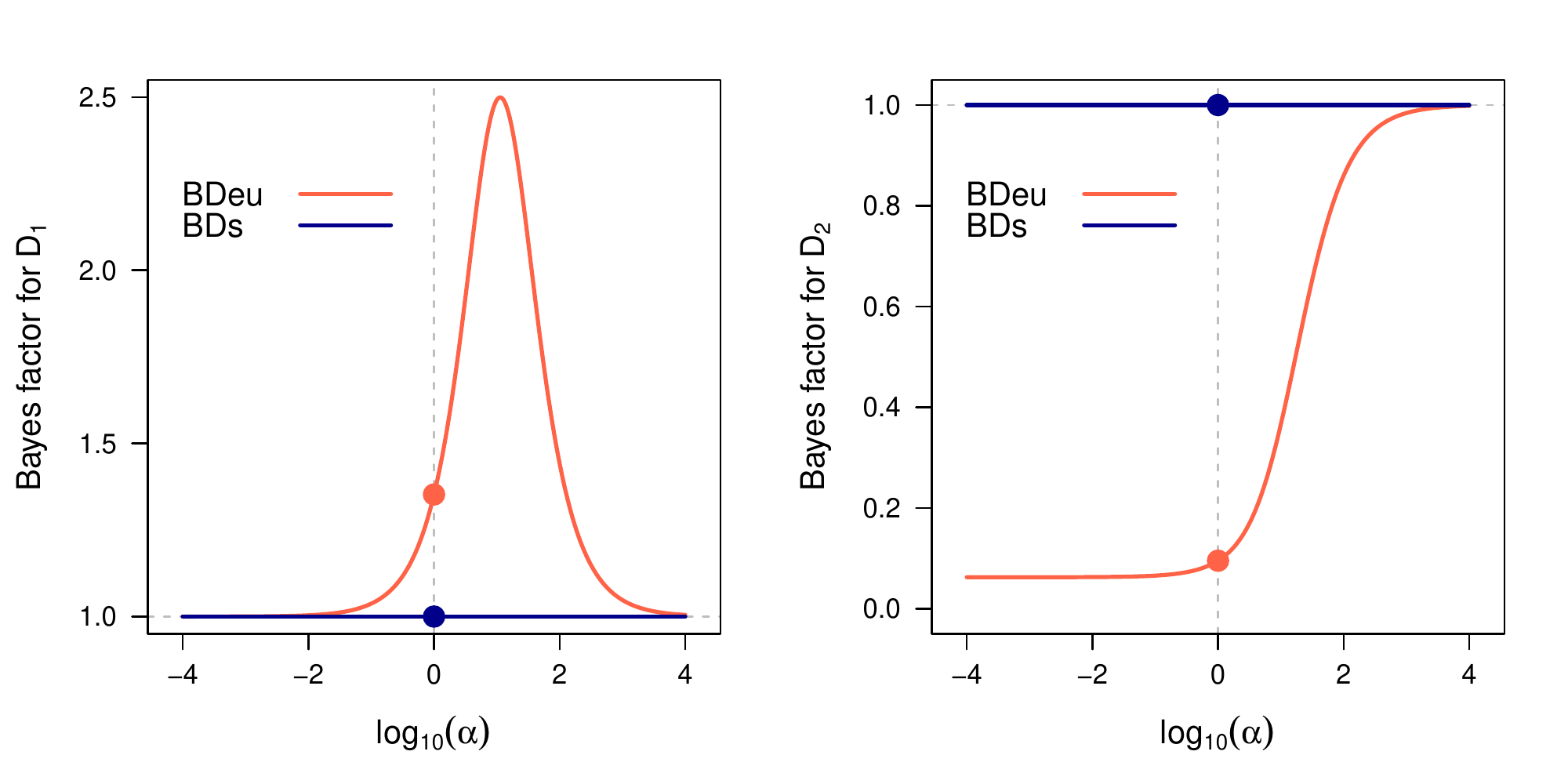}
\caption{The Bayes factors for $\Gm$ versus $\Gp$ computed using BDeu and
  BDs for Example \ref{ex:singular} (left panel) and Example \ref{ex:nonsingular}
  (right panel in orange) and dark blue, respectively. The bullet points
 correspond to the values observed for $\alpha = 1$.}
\label{fig:bf}
\end{center}
\end{figure}

\section{Bayesian Structure Learning and Entropy}
\label{sec:maxent}

Shannon's classic definition of entropy for a multinomial random variable
$X \sim \Mul(\bpi)$ with a fixed, finite set of states (alphabet) $\mathcal{A}$
is
\begin{equation*}
  \H(X; \bpi) = \E(-\log\Prob(X)) = -\sum_{a \in \mathcal{A}} \pi_a \log \pi_a
\end{equation*}
where the probabilities $\pi_a$ are typically estimated with the empirical
frequency of each $a$ in $\D$, leading to the \emph{empirical entropy
estimator}. Its properties are detailed in canonical information theory books
such as \citet{mackay} and \citet{rissanen}, and it has often been used in BN
structure learning \citep{lam,suzuki15}. However, in this paper we will focus
on Bayesian entropy estimators, for two reasons. Firstly, they are a natural
choice when studying the properties of BD scores since they are Bayesian in
nature; and having the same probabilistic assumptions (including the choice of
prior distribution) for the BD scores and for the corresponding entropy
estimators makes it easy to link their properties. Secondly, Bayesian entropy
estimators have better theoretical and empirical properties than the empirical
estimator \citep{shrinkent,hnsb}.

Starting from \mref{eq:parents}, for a BN we can write
\begin{equation*}
  \HG(\X; \Theta) = \sum_{i = 1}^N \HG(X_i; \T).
\end{equation*}
where $\HG(X_i; \T)$ is the entropy of $X_i$ given its parents $\PXiG$. The
marginal posterior expectation of $\H^{\G}(X_i; \T)$ with respect to $\T$ given
the data can then be expressed as
\begin{equation*}
  \E\left(\HG(X_i) \given \D\right) = \int \HG(X_i; \T) \Prob(\T \given \D) \dd\T
\end{equation*}
where we use $\D$ to refer specifically to the observed values for $X_i$ and
$\PXiG$ with a slight abuse of notation. We can then introduce a $\Dir(\aijk)$
prior over $\T$ with
\begin{equation*}
  \Prob(\T \given \D) = \int \Prob(\T \given \D, \aijk) \Prob(\aijk \given \D) \dd\aijk,
\end{equation*}
which leads to
\begin{align}
  \E\left(\HG(X_i) \given \D\right)
    &= \iint \HG(X_i; \T) \Prob(\T \given \D, \aijk) \Prob(\aijk \given \D) \dd\aijk \dd\T \notag \\
    &\propto \int \E\left(\HG(X_i) \given \D, \aijk\right) \Prob(\D \given \aijk) \Prob(\aijk) \dd\aijk,
\label{eq:postent}
\end{align}
where $\Prob(\aijk)$ is a hyperprior distribution over the space of the
Dirichlet priors, identified by their parameter sets $\{\aijk\}$.

The first term on the right hand-side of \mref{eq:postent} is the posterior
expectation of
\begin{align}
  &\HG(X_i | \D, \aijk) = - \sum_{j = 1}^{q_i} \sum_{k = 1}^{r_i}
                        \pijk^{(\aijk)} \log \pijk^{(\aijk)}&
  &\text{with}&
  &\pijk^{(\aijk)} = \frac{\aijk + n_{ijk}}{\aij + n_{ij}}
\label{eq:mpostent}
\end{align}
and has closed form
\begin{multline}
  \E\left(\HG(X_i) \given \D, \aijk\right) = \\
    \sum_{j = 1}^{q_i} \left[
        \psi_0(\aij + n_{ij} + 1) -
        \sum_{k = 1}^{r_i} \frac{\aijk + n_{ijk}}{\aij + n_{ij}}
                             \psi_0(\aijk + n_{ijk} + 1)
    \right]
\label{eq:empent}
\end{multline}
as shown in \citet{hnsb} and \citet{archer}, with $\psi_0(\cdot)$ denoting the
digamma function. The second term follows a \emph{Dirichlet-multinomial
distribution} \citep[also known as \emph{multivariate Polya};][]{kotz} with
density
\begin{equation}
  \Prob(\D \given \aijk)
    = \prod_{j = 1}^{q_i} \frac{n_{ij}! \, \Gamma(\aij)}{\Gamma(\aijk)^{r_i}
                                           \Gamma(n_{ij} + \aij)}
      \prod_{k = 1}^{r_i}
      \frac{\Gamma(n_{ijk} + \aijk)}{n_{ijk}!},
\label{eq:multdir}
\end{equation}
since
\begin{align*}
  \Prob(\D \given \aijk) = \int \Prob(\D \given \T) \Prob(\T \given \aijk) \dd\T
\end{align*}
where $\Prob(\D \given \T)$ follows a multinomial distribution and
$\Prob(\T \given \aijk)$ is a conjugate Dirichlet prior. Rearranging terms in
\mref{eq:multdir} we find that
\begin{multline}
  \Prob(\D \given \aijk)
    = \prod_{j = 1}^{q_i} \frac{n_{ij}!}{\prod_{k = 1}^{r_i} n_{iijk}!} \cdot \\
      \prod_{j = 1}^{q_i} \frac{\Gamma(\aij)}{\Gamma(n_{ij} + \aij)}
      \prod_{k = 1}^{r_i} \frac{\Gamma(n_{ijk} + \aijk)}{\Gamma(\aijk)}
    \propto \BD\left(\XPi^\G; \alpha\right)
\label{eq:probbd}
\end{multline}
making the link between BD scores and entropy explicit. Unlike \mref{eq:probbd},
BD has a prequential formulation \citep{dawid} which focuses on the sequential
prediction of future events \citep{engineering}. For this reason it considers
observations as coming in a specific temporal order and it does not include a
multinomial coefficient, which we will drop in the remainder of the paper.
Therefore,
\begin{equation}
  \E\left(\HG(X_i) \given \D\right) =
     \int \E\left(\HG(X_i) \given \D, \aijk\right) \BD\left(\XPi^\G; \alpha\right) \Prob(\aijk) \dd\aijk,
\label{eq:entbd}
\end{equation}
and is determined by three components: the posterior expected entropy of
$\XPiG$ under a $\Dir(\aijk)$ prior, the BD score term for $\XPiG$, and the
hyperprior over the space of the Dirichlet priors.

This definition of the expected entropy associated with the structure $\G$ of a
BN is very general and encompasses the entropies associated with all the BD
scores as special cases. In particular, the entropy associated with K2, BDJ,
BDeu and BDs can be obtained by giving $\Prob(\aijk) = 1$ to the single
set of $\aijk$ associated with the corresponding prior, leading to
\begin{equation*}
  \E\left(\HG(X_i) \given \D\right) =
    \E\left(\HG(X_i) \given \D, \aijk\right) \BD\left(\XPi^\G; \alpha\right);
\label{eq:simple}
\end{equation*}
and similarly for BDla
\begin{equation*}
  \E\left(\HG(X_i) \given \D\right) = \frac{1}{|S_L|}
    \sum_{s \in S_L} \E\left(\HG(X_i) \given \D, \aijk^s\right) \BD\left(\XPi^\G; s\right).
\end{equation*}

\section{The Posterior Marginal Entropy}
\label{sec:expent}

The posterior expectation of the entropy for a given $\Dir(\aijk)$ prior in
\mref{eq:empent}, despite having a form that looks very different from the
marginal posterior entropy in \mref{eq:mpostent}, can be written in terms of
the latter as we show in the following lemma.

\begin{lemma}
  \begin{equation*}
    \E\left(\HG(X_i) \given \D; \aijk\right) \approx
      \HG(X_i \given \D, \aijk) - \sum_{j = 1}^{q_i} \frac{r_i - 1}{2(\aij + n_{ij})}.
  \end{equation*}
  \label{aic}
\end{lemma}

\begin{proof}{\em of Lemma \ref{aic}.}
  Combining $\psi_0(z + 1) = \psi_0(z) + 1/z$ with
  $\psi_0(z) = \log(z) - 1/(2z) + o(z^{-2})$ from \citet{anderson}, we can write
  $\psi_0(z + 1) = \log(z) + 1/(2z) + o(z^{-2})$. Dropping the remainder term
  $o(z^{-2})$ we approximate $\psi_0(z + 1) \approx \log(z) + 1/(2z)$, which leads to
  \begin{align*}
    \E&\left(\HG(X_i) \given \D, \aijk\right) = \\
      &= \sum_{j = 1}^{q_i} \left[
          \psi_0(\aij + n_{ij} + 1) -
          \sum_{k = 1}^{r_i} \frac{\aijk + n_{ijk}}{\aij + n_{ij}}
                               \psi_0(\aijk + n_{ijk} + 1)
       \right] \\
      &\approx \sum_{j = 1}^{q_i} \left[
          \log(\aij + n_{ij}) + \frac{1}{2(\aij + n_{ij})} - \right. \\
      &\left. \qquad \qquad
          \sum_{k = 1}^{r_i} \frac{\aijk + n_{ijk}}{\aij + n_{ij}}
                               \left( \log(\aijk + n_{ijk}) + \frac{1}{2(\aijk + n_{ijk})} \right)
      \right] \\
      &= -\sum_{j = 1}^{q_i}
          \sum_{k = 1}^{r_i} \frac{\aijk + n_{ijk}}{\aij + n_{ij}}
                               \log\left(\frac{\aijk + n_{ijk}}{\aij + n_{ij}}\right)
       - \sum_{j = 1}^{q_i} \frac{r_i - 1}{2(\aij + n_{ij})} \\
      &= \HG(X_i \given \D, \aijk) - \sum_{j = 1}^{q_i} \frac{r_i - 1}{2(\aij + n_{ij})}.
  \end{align*} \qed
\end{proof}

Therefore, $\E(\HG(X_i) \given \D; \aijk)$ is well approximated by the marginal
posterior entropy $\HG(X_i | \D, \aijk)$ from \mref{eq:mpostent} plus a bias
term that depends on the augmented counts $\aij + n_{ij}$ for the $q_i$
configurations of $\PXiG$. A similar result was derived in \citet{miller} for
the empirical entropy estimator and is the basis of the Miller-Madow entropy
estimator.

\section{BDeu and the Principle of Maximum Entropy}
\label{sec:bd-maxent}

The \emph{maximum relative entropy} principle \citep[ME;][]{shore,skilling,caticha2}
states that we should choose a model that is consistent with our knowledge of
the phenomenon we are modelling and that introduces no unwarranted information.
In the context of probabilistic learning this means choosing the model that has
the largest possible entropy for the data, which will encode the probability
distribution that best reflects our current knowledge of $\X$ given by $\D$. In
the Bayesian setting in which BD scores are defined, we then prefer a DAG $\Gp$
over a second DAG $\Gm$ if
\begin{equation}
  \E\left(\H^{\Gm}(\X) \given \D\right) \leqslant \E\left(\H^{\Gp}(\X) \given \D\right)
\label{eq:maxent}
\end{equation}
because these estimates of entropy incorporate all our knowledge including that
encoded in the prior and in the hyperprior. The resulting formulation of ME
represents a very general approach that includes Bayesian posterior estimation
as a particular case \citep{caticha}; which is intuitively apparent since the
expected posterior entropy in \mref{eq:entbd} is proportional to BD. Furthermore,
ME can also be seen as a particular case of the MDL principle \citep{feder}.

\cite{suzuki16} defined \emph{regular} those BD scores that prefer simpler BNs
that have smaller empirical entropies and few arcs:
\begin{multline*}
  \H\left(\XPiGm; \piijk\right) \leqslant \H\left(\XPiGp; \piijk\right),
    \PXiGm \subset \PXiGp \Rightarrow \\
  \BD\left(\XPiGm; \alpha_i\right) \geqslant \BD\left(\XPiGp; \alpha_i\right).
\end{multline*}
For large sample sizes, the probabilities $\pijk^{(\aijk)}$ from
\mref{eq:mpostent} used in the posterior entropy estimators converge to the
empirical frequencies used in the empirical entropy estimator, making the above
asymptotically equivalent to
\begin{multline*}
  \H\left(\XPiGm; \aijk\right) \leqslant \H\left(\XPiGp; \aijk\right),
    \PXiGm \subset \PXiGp \Rightarrow \\
  \BD\left(\XPiGm; \alpha_i\right) \geqslant \BD\left(\XPiGp; \alpha_i\right)
\end{multline*}
and connecting DAGs with the highest BD scores with those that minimise the
marginal posterior entropy from \mref{eq:mpostent}. However, we prefer to study
BDeu and its prior using ME as defined in \mref{eq:maxent} for two reasons.
Firstly, posterior expectations are widely considered to be superior to MAP
estimates in the literature \citep{berger}, as has been specifically shown for
entropy in \citet{hnsb}. Secondly, ME directly incorporates the information
encoded in the prior and in the hyperprior, without relying on large samples to
link the empirical entropy (which depends on $\X, \Theta$) with the BD scores
(which depend on $\X, \api$ and integrate $\Theta$ out).

Without loss of generality, we consider again the simple case in which $\Gm$ and
$\Gp$ differ by a single arc, so that only one local distribution differs
between the two DAGs. For BDeu, $\aijk = \alpha / (r_i q_i)$ and substituting
\mref{eq:simple} in \mref{eq:maxent} we get
\begin{multline}
  \E\left(\H^{\Gm}(X_i) \given \D, \aijk\right) \BDeu\left(\XPiGm; \alpha\right) \leqslant \\
  \E\left(\H^{\Gp}(X_i) \given \D, \aijk\right) \BDeu\left(\XPiGp; \alpha\right).
\label{eq:bdeucomp}
\end{multline}
If the sample $\D$ is sparse, some configurations of the variables will not be
observed and the effective imaginary sample size may be smaller for $\Gp$ than
for $\Gm$ like in Examples \ref{ex:singular} and \ref{ex:nonsingular}. As a
result, when we compare a $\Gm$ for which we observe all configurations of
$\PXiGm$ with a $\Gp$ for which we do not observe some configurations of
$\PXiGp$, instead of \mref{eq:bdeucomp} we are actually using
\begin{multline}
  \E\left(\H^{\Gm}(X_i) \given \D, \aijk\right) \BDeu\left(\XPiGm; \alpha\right) \leqslant \\
  \E\left(\H^{\Gp}(X_i) \given \D, \aijk\right) \BDeu\left(\XPiGp; \alpha (\tq / q_i)\right)
\label{eq:bdeucomp2}
\end{multline}
which is different from \mref{eq:maxent} and thus not consistent with ME. It is
not correct from a Bayesian perspective either, because $\Gm$ and $\Gp$ are
compared using marginal likelihoods arising from different priors; as expected
since we know from \citet{caticha} that Bayesian posterior inference is a
particular case of ME. From the perspective of ME, those priors carry different
information on the $\T$. They incorrectly express different levels of belief in
the uniform prior underlying BDeu as a consequence of the difference in their
effective imaginary sample sizes, even though they were meant to express the
same level of belief for all possible DAGs. This may bias the entropy (which is,
after all, the expected value of the information on $\X$ and on the $\T$) of
$\Gp$ compared to that of $\Gm$ in \mref{eq:bdeucomp2} and lead to choosing DAGs
which would not be chosen by ME. We would like to stress that this scenario is
not uncommon; on the contrary, such a model comparison is almost guaranteed to
take place when the data are sparse. As structure learning explores more and
more DAGs to identify an optimal one, it will inevitably consider DAGs with
unobserved parent configurations, either because they are too dense or because
those parent sets are not well supported by the few observed data points.

In particular, if some $n_{ij} = 0$ then the posterior expected entropy of
$\XPi^{\Gp}$ becomes
\begin{multline*}
\E\left(\H^{\Gp}(X_i) \given \D, \aijk\right) = \\
  \sum_{j : n_{ij} = 0} \left[
          \psi_0(r_ i \aijk + 1) -
          \sum_{k = 1}^{r_i} \frac{\cancel{\aijk}}{r_i \cancel{\aijk}}
                               \psi_0(\aijk + 1)
       \right] + \\
  \sum_{\jpo} \left[
          \psi_0(r_i \aijk + n_{ij} + 1) -
          \sum_{k = 1}^{r_i} \frac{\aijk + n_{ijk}}{r_i \aijk + n_{ij}}
                               \psi_0(\aijk + n_{ijk} + 1)
       \right]
\end{multline*}
where the first term collects the conditional entropies corresponding to the
$q_i - \tq$ unobserved parent configurations, for which the posterior
distribution coincides with the prior:
\begin{multline}
  \sum_{j : n_{ij} = 0} \left[
          \psi_0(r_ i \aijk + 1) -
          \sum_{k = 1}^{r_i} \frac{1}{r_i}
                               \psi_0(\aijk + 1)
       \right] \approx \\
-\sum_{j : n_{ij} = 0}
          \sum_{k = 1}^{r_i} \frac{\cancel{\aijk}}{r_i \cancel{\aijk}}
                               \log\left(\frac{\cancel{\aijk}}{r_i \cancel{\aijk}}\right)
       - \sum_{j = 1}^{q_i} \frac{r_i - 1}{2\aij} =
(q_i - \tq) \left[ -\log\frac{1}{r_i} - \frac{r_i - 1}{2\aij} \right].
\label{eq:vanishing}
\end{multline}

By definition, the uniform distribution has the maximum possible entropy; the
posteriors we would estimate if we could observe samples for those
configurations of the $\PXiGp$ would almost certainly have a smaller entropy.
At the same time, the entropies in the second term are smaller than what
they would be if we only considered the $\tq$ observed parent configurations,
because $\aijk = \alpha / (r_i q_i) < \alpha / (r_i \tq)$ means that posterior
densities deviate more from the uniform distribution. These two effects,
however, do not necessarily balance each other out; as we can see by revisiting
Examples \ref{ex:singular} and \ref{ex:nonsingular} below it is possible to
incorrectly choose $\Gp$ over $\Gm$.

\setcounter{example}{0}
\begin{example}{\em (Continued)}
The empirical posterior entropies for $\Gm$ and $\Gp$ are
\begin{equation*}
\H(\XP^{\Gm}) = \H(\XP^{\Gp}) = 4 \left[
  - 0 \log 0 -
	1 \log 1\right] = 0
\end{equation*}
by convention, but the posterior entropies differ:
\begin{align*}
  \H(\XP^{\Gm}; \alpha) &= 4 \left[
    - \frac{0 + \sfrac{1}{8}}{3 + \sfrac{1}{4}} \log \frac{0 + \sfrac{1}{8}}{3 + \sfrac{1}{4}} -
      \frac{3 + \sfrac{1}{8}}{3 + \sfrac{1}{4}} \log \frac{3 + \sfrac{1}{8}}{3 + \sfrac{1}{4}} \right] \\
    &= 0.652,\\
  \H(\XP^{\Gp}; \alpha) &= 4 \left[
    - \frac{0 + \sfrac{1}{16}}{3 + \sfrac{1}{8}} \log \frac{0 + \sfrac{1}{16}}{3 + \sfrac{1}{8}} -
      \frac{3 + \sfrac{1}{16}}{3 + \sfrac{1}{8}} \log \frac{3 + \sfrac{1}{16}}{3 + \sfrac{1}{8}} \right] \\
    &= 0.392.
\end{align*}
The expected posterior entropies for $\Gm$ and $\Gp$ are
\begin{align*}
  \E&\left(\H^{\Gm}(X) \given \D, \frac{1}{8}\right) = \\
    &= 4 \left[
      \psi_0(\sfrac{1}{4} + 3 + 1)
        - \frac{0 + \sfrac{1}{8}}{3 + \sfrac{1}{4}} \psi_0(\sfrac{1}{8} + 0 + 1)
        - \frac{3 + \sfrac{1}{8}}{3 + \sfrac{1}{4}} \psi_0(\sfrac{1}{8} + 3 + 1)
      \right] \\
    &= 0.3931, \\
  \E&\left(\H^{\Gp}(X) \given \D, \frac{1}{16}\right) = \\
    &= 4 \left[
      \psi_0(\sfrac{1}{8} + 3 + 1) - \frac{0 + \sfrac{1}{16}}{3 + \sfrac{1}{8}} \psi_0(\sfrac{1}{16} + 0 + 1)
        - \frac{3 + \sfrac{1}{16}}{3 + \sfrac{1}{8}} \psi_0(\sfrac{1}{16} + 3 + 1)
      \right] + \\
    &\qquad 4 \left[
      \psi_0(\sfrac{1}{8} + 0 + 1) - \frac{0 + \sfrac{1}{16}}{0 + \sfrac{1}{8}} \psi_0(\sfrac{1}{16} + 0 + 1)
        - \frac{3 + \sfrac{1}{16}}{0 + \sfrac{1}{8}} \psi_0(\sfrac{1}{16} + 0 + 1)
    \right] \\
    &= 0.5707.
\end{align*}
Therefore, substituting these values in \mref{eq:bdeucomp2},
\begin{multline*}
  \E\left(\H^{\Gm}(X) \given \D\right) = 0.3931 \cdot 0.0326 = 0.0128 < \\
    0.0252 = 0.5707 \cdot 0.0441 = \E\left(\H^{\Gp}(X) \given \D\right);
\end{multline*}
and we would choose $\Gp$ over $\Gm$ even though we only observe $\tq = 4$
configurations of $\PX^{\Gp}$ out of $8$, and the sample frequencies are
identical for those configurations. The data contribute the same information
to the posterior expected entropies; both $\XPGm$ and $\XPGp$ have empirical
entropy equal to zero. The difference must then arise because of the priors:
both $\Theta_X \given \PX^{\Gm}$ and $\Theta_X \given \PX^{\Gp}$ follow a
uniform Dirichlet prior, but in the former $\alpha = 1$ and in the latter
$\alpha = \fh$ because $\tq = 4 < 8 = q_i$. A consistent model comparison
requires that all models are evaluated with the same prior, which clearly is not
the case in this example. \qed
\end{example}

\begin{example}{\em (Continued)}
The conditional distributions of  $\XP^{\Gm}$ and $\XP^{\Gp}$ both have the same
(positive) empirical entropy:
\begin{equation*}
  \H(\XP^{\Gm}) = \H(\XP^{\Gp}) = 4 \left[
    - \frac{1}{3} \log \frac{1}{3} -
      \frac{2}{3} \log \frac{2}{3}\right] = 2.546.
\end{equation*}
However, their posterior entropies are different:
\begin{align*}
  \H(\XP^{\Gm}; \alpha) &= 4 \left[
    - \frac{1 + \sfrac{1}{8}}{3 + \sfrac{1}{4}} \log \frac{1 + \sfrac{1}{8}}{3 + \sfrac{1}{4}} -
      \frac{2 + \sfrac{1}{8}}{3 + \sfrac{1}{4}} \log \frac{2 + \sfrac{1}{8}}{3 + \sfrac{1}{4}} \right] \\
    &= 2.580, \\
  \H(\XP^{\Gp}; \alpha) &= 4 \left[
    - \frac{1 + \sfrac{1}{16}}{3 + \sfrac{1}{8}} \log \frac{1 + \sfrac{1}{16}}{3 + \sfrac{1}{8}} -
      \frac{2 + \sfrac{1}{16}}{3 + \sfrac{1}{8}} \log \frac{2 + \sfrac{1}{16}}{3 + \sfrac{1}{8}} \right] \\
    &= 2.564;
\end{align*}
and the respective posterior expected entropies are:
\begin{align*}
  \E&\left(\H^{\Gm}(X) \given \D, \frac{1}{8}\right) = \\
    &= 4 \left[
      \psi_0(\sfrac{1}{4} + 3 + 1)
        - \frac{1 + \sfrac{1}{8}}{3 + \sfrac{1}{4}} \psi_0(\sfrac{1}{8} + 1 + 1)
        - \frac{2 + \sfrac{1}{8}}{3 + \sfrac{1}{4}} \psi_0(\sfrac{1}{8} + 2 + 1)
      \right] \\
    &= 2.066, \\
  \E&\left(\H^{\Gp}(X) \given \D, \frac{1}{16}\right) = \\
    &= 4 \left[
      \psi_0(\sfrac{1}{8} + 3 + 1) - \frac{0 + \sfrac{1}{16}}{3 + \sfrac{1}{8}} \psi_0(\sfrac{1}{16} + 1 + 1)
        - \frac{3 + \sfrac{1}{16}}{3 + \sfrac{1}{8}} \psi_0(\sfrac{1}{16} + 2 + 1)
      \right] + \\
    &\qquad 4 \left[
      \psi_0(\sfrac{1}{8} + 0 + 1) - \frac{0 + \sfrac{1}{16}}{0 + \sfrac{1}{8}} \psi_0(\sfrac{1}{16} + 0 + 1)
        - \frac{3 + \sfrac{1}{16}}{0 + \sfrac{1}{8}} \psi_0(\sfrac{1}{16} + 0 + 1)
    \right] \\
    &= 4.069.
\end{align*}
Therefore, substituting these values in \mref{eq:bdeucomp2} leads to
\begin{multline*}
  \E\left(\H^{\Gm}(X) \given \D\right) = 2.066 \cdot 3.906\e{-7} = 8.071\e{-7} > \\
    1.514\e{-7} = 4.069 \cdot 3.721\e{-8} = \E\left(\H^{\Gp}(X) \given \D\right).
\end{multline*}
Even though we choose $\Gm$ over $\Gp$, we still express a preference for one of
the DAGs even though the information in the data is the same; which confirms
that the difference is attributable to the prior. \qed
\end{example}

\begin{figure}[t]
\begin{center}
\includegraphics[width=0.9\textwidth]{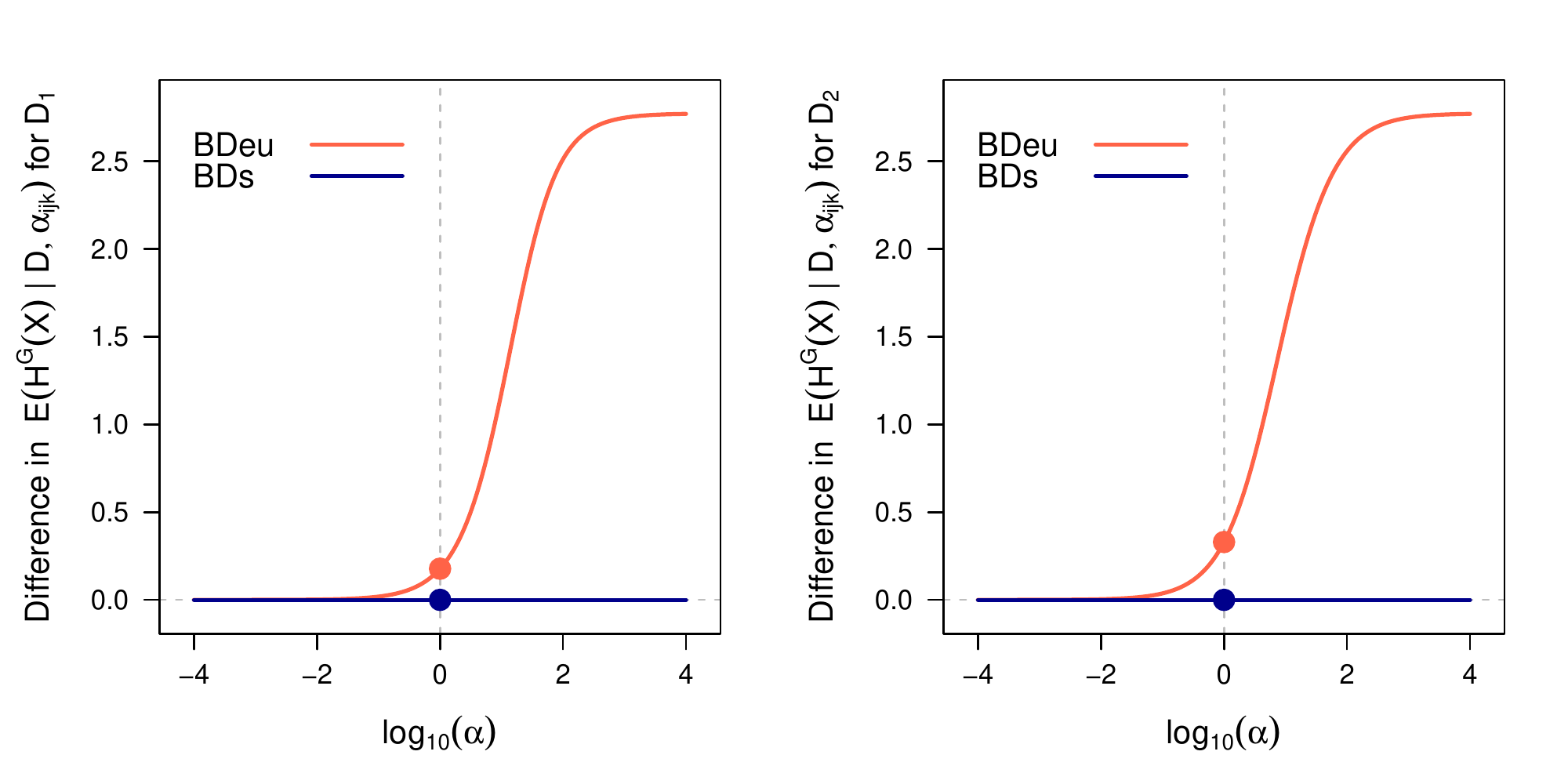}
\caption{The difference between $\E\left(\H^{\Gm}(X) \given \D, \aijk \right)$
  and $\E\left(\H^{\Gp}(X) \given \D, \aijk \right)$ computed using BDeu and
  BDs for Example \ref{ex:singular} (left panel) and Example \ref{ex:nonsingular}
  (right panel) in orange and dark blue, respectively. The bullet points
 correspond to the values observed for $\alpha = 1$.}
\label{fig:ement}
\end{center}
\end{figure}

\begin{figure}[t]
\begin{center}
\includegraphics[width=0.9\textwidth]{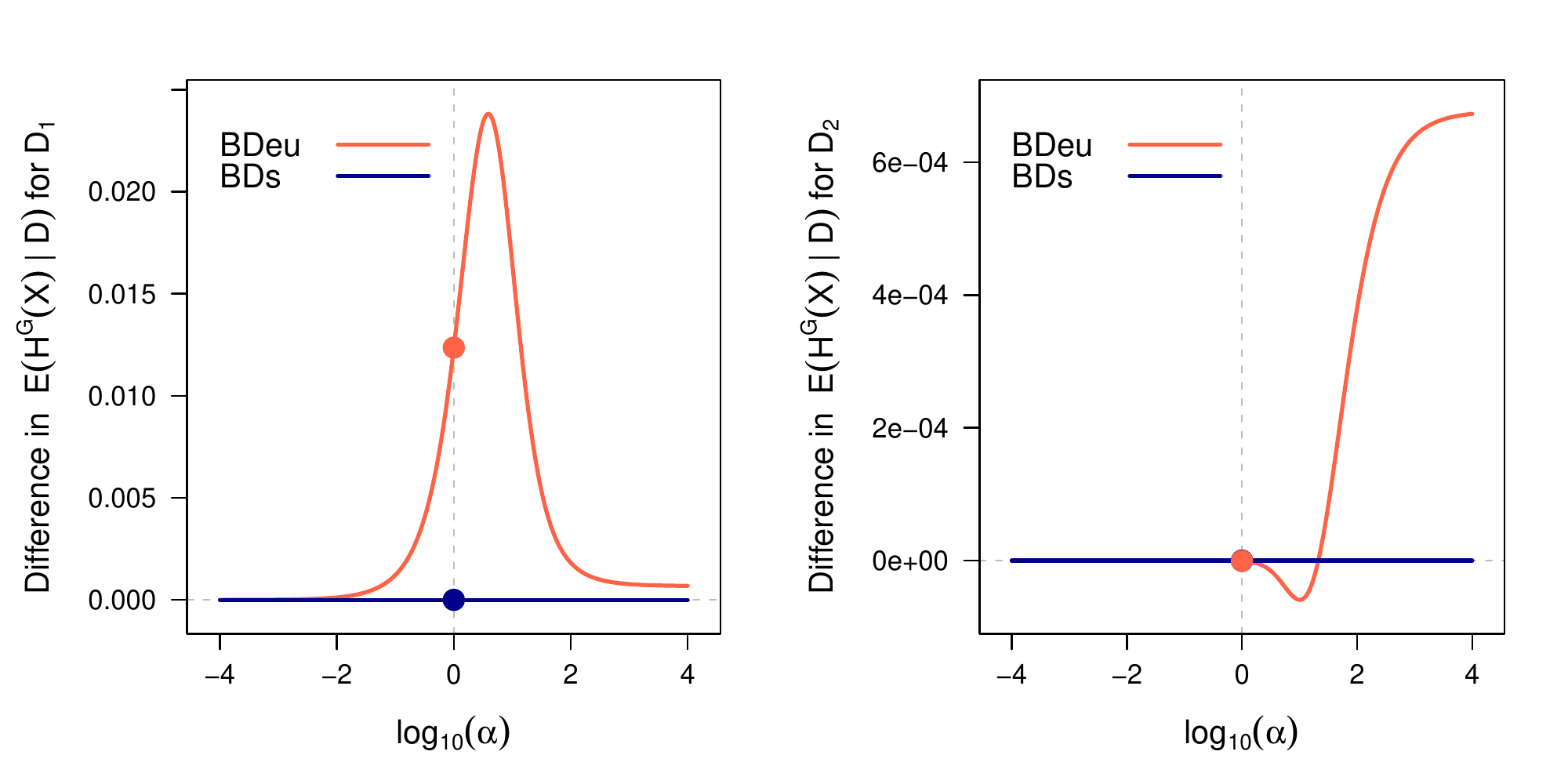}
\caption{The difference between $\E\left(\H^{\Gm}(X) \given \D\right)$
  and $\E\left(\H^{\Gp}(X) \given \D\right)$ computed using BDeu and
  BDs for Example \ref{ex:singular} (left panel) and Example \ref{ex:nonsingular}
  (right panel) in orange and dark blue, respectively. The bullet points
 correspond to the values observed for $\alpha = 1$.}
\label{fig:ement2}
\end{center}
\end{figure}

Based on these results and the examples above, we state the following theorem.

\begin{theorem}
  Using BDeu and the associated uniform prior over the parameters of a BN for
  structure learning violates the maximum relative entropy principle if any
  candidate parent configuration of any node is not observed in the data.
\end{theorem}

This is not the case for BDs, because its piecewise uniform prior preserves the
imaginary sample size even when $\tq < q_i$; and because it prevents the
posterior entropy from inflating by allowing the $\tq$ terms corresponding to
the $n_{ij} = 0$ to simplify. Assuming $\aijk = 0$ in \mref{eq:vanishing}
implies
\begin{equation*}
  \sum_{j : n_{ij} = 0} \left[
          \psi_0(1) -
          \sum_{k = 1}^{r_i} \frac{1}{r_i} \psi_0(1)
       \right] = \psi_0(1) - \psi_0(1) = 0.
\end{equation*}

\setcounter{example}{0}
\begin{example}{\em (Continued)}
If we compare $\XPGm$ and $\XPGp$ under the prior assumed by BDs, we have that
$\aijk = \sfrac{1}{8}$ for both $\XPGm$ and the $\tq$ observed parent
configurations in $\XPGp$. Then their posterior expected entropies are
\begin{align*}
  \E&\left(\H^{\Gm}(X) \given \D, \frac{1}{8}\right) =
  \E\left(\H^{\Gp}(X) \given \D, \frac{1}{8}\right) = \\
    &= 4 \left[
      \psi_0(\sfrac{1}{4} + 3 + 1)
        - \frac{0 + \sfrac{1}{8}}{3 + \sfrac{1}{4}} \psi_0(\sfrac{1}{8} + 0 + 1)
        - \frac{3 + \sfrac{1}{8}}{3 + \sfrac{1}{4}} \psi_0(\sfrac{1}{8} + 3 + 1)
      \right] \\
    &= 0.3931
\end{align*}
and substituting these values in \mref{eq:bdeucomp}
\begin{multline*}
  \E\left(\H^{\Gm}(X) \given \D\right) = 0.3931 \cdot 0.0326 = 0.0128 = \\
    0.0128 = 0.3931 \cdot 0.0326 = \E\left(\H^{\Gp}(X) \given \D\right).
\end{multline*}
ME does not express a preference for either $\Gm$ or $\Gp$; since we have
observed above that the data contribute exactly the same information for both
DAGs, the same must be true for the prior associated with BDs.

A side effect of not violating ME is that the choice between $\Gm$ and $\Gp$ is
no longer sensitive to the value of $\alpha$; we can see from the left panels
of Figures \ref{fig:ement} and \ref{fig:ement2} that both the difference
between $\E\left(\H^{\Gm}(X) \given \D, \frac{1}{8}\right)$ and
$\E\left(\H^{\Gp}(X) \given \D, \frac{1}{8}\right)$ and the difference between
$\E\left(\H^{\Gm}(X) \given \D\right)$ and $\E\left(\H^{\Gp}(X) \given \D\right)$
are equal to zero for all $\alpha$. \qed
\end{example}

\begin{example}{\em (Continued)}
Again $\aijk = \sfrac{1}{8}$ for both $\XPGm$ and the $\tq$ observed parent
configurations in $\XPGp$, so
\begin{align*}
  \E&\left(\H^{\Gm}(X) \given \D, \frac{1}{8}\right) =
  \E\left(\H^{\Gp}(X) \given \D, \frac{1}{8}\right) = \\
    &= 4 \left[
      \psi_0(\sfrac{1}{4} + 3 + 1)
        - \frac{1 + \sfrac{1}{8}}{3 + \sfrac{1}{4}} \psi_0(\sfrac{1}{8} + 1 + 1)
        - \frac{2 + \sfrac{1}{8}}{3 + \sfrac{1}{4}} \psi_0(\sfrac{1}{8} + 2 + 1)
      \right] \\
    &= 2.066
\end{align*}
which leads to
\begin{multline*}
   \E\left(\H^{\Gm}(X) \given \D\right) = 2.066 \cdot 3.906\e{-7} = 8.071\e{-7} = \\
      8.071\e{-7} = 2.066 \cdot 3.906\e{-7} = \E\left(\H^{\Gp}(X) \given \D\right).
\end{multline*}
Again we can see from the right panels of Figures \ref{fig:ement} and
\ref{fig:ement2} that the choice between $\Gm$ and $\Gp$ is no longer sensitive
to the choice of $\alpha$; and $\Gp$ is never preferred to $\Gm$. This contrasts
especially the behaviour of BDeu in Figure \ref{fig:ement2}, where
$\E\left(\H^{\Gp}(X) \given \D\right)$ can be both larger and smaller than
$\E\left(\H^{\Gm}(X) \given \D\right)$ for different values of $\alpha$. \qed
\end{example}

It is easy to show that the theorem we just stated does not apply to K2 or BDJ,
since under their priors $\aijk$ is not a function of $q_i$; but it does apply
to BDla since its formulation is essentially a mixture of BDeu scores.

\section{Conclusions and Discussion}

Bayesian network learning follows an inherently Bayesian workflow in which we
first learn the structure of the DAG $\G$ from a data set $\D$, and then we
learn the values of the parameters $\T$ given $\G$. In this paper we studied
the properties of the Bayesian posterior scores used to estimate
$\Prob(\G \given \D)$ and to learn the $\G$ that best fits the data. For
discrete Bayesian networks, these scores are Bayesian Dirichlet (BD) marginal
likelihoods that assume different Dirichlet priors for the $\T$ and, in the most
general formulation, a hyperprior over the hyperparameters $\aijk$ of the
prior. We focused on the most common BD score, BDeu, which assumes a uniform
prior over each $\T$; and we studied the impact of that prior on structure
learning from a Bayesian and an information theoretic perspective. After deriving
the form of the posterior expected entropy for $\G$ given $\D$, we found that
BDeu may select models in a way that violates the maximum relative entropy
principle. Furthermore, we showed that it produces Bayes factors that are very
sensitive to the choice of the imaginary sample size. Both issues are related to
the uniform prior assumed by BDeu for the $\T$, and can lead to the selection
of overly dense DAGs when the data are sparse. In contrast, the BDs score
proposed in \citet{jmlr16} does not, even though it converges to BDeu
asymptotically; and neither do other BD scores in the literature. In the
simulation study we performed in \citet{jmlr16}, we found that BDs leads to
more accurate structure learning; hence we recommend its use over BDeu for
sparse data.



\end{document}